\documentclass[reqno]{amsart}
  \usepackage{mathabx}
\usepackage{cite}
\allowdisplaybreaks

 \date{\today}
\theoremstyle{plain}
\newtheorem{thm}{Theorem}[section]

\newtheorem{lem}[thm]{Lemma}
\newtheorem{prop}[thm]{Proposition}

\theoremstyle{definition}

\theoremstyle{remark}
\newtheorem{rem}{Remark}[section]
\numberwithin{equation}{section}
\renewcommand{\theequation}{\thesection.\arabic{equation}}

\renewcommand{\u}{{\bf u}}
\renewcommand{\v}{{\bf v}}
\renewcommand{\H}{{\bf H}}
\renewcommand{\j}{{\bf J}}
\renewcommand{\i}{{\bf E}}
\renewcommand{\v}{{\vvvert}}
\newcommand{\dv}{{\rm div\,}}
\newcommand{\dif}{{\rm d}}
\newcommand{\cu}{{\rm curl\,}}

\newcommand{\U}{{\mathbf U}}
\newcommand{\F}{{\mathbf F}}
\newcommand{\G}{{\mathbf G}}
\newcommand{\W}{{\mathbf W}}
\newcommand{\D}{{\mathbf D}}
\newcommand{\A}{{\mathbf A}}
\newcommand{\s}{{\mathbf S}}

\newcommand{\p}{{\partial}}
\newcommand{\pa}{{\partial^\alpha_x}}
\newcommand{\ep}{{\epsilon}}

\begin{document}

\title[Euler-Maxwell system to compressible MHD equations]
 {zero  dielectric constant limit to the non-isentropic compressible
Euler-Maxwell system}

\author{Song Jiang}
\address{Institute of Applied Physics and Computational Mathematics, P.O. Box 8009, Beijing 100088, P.R. China}
 \email{jiang@iapcm.ac.cn}

\author [Fucai Li] {Fucai Li$^*$}%
\thanks{$^*$Corresponding author}
\address{Department  of Mathematics, Nanjing University, Nanjing 210093, P.R. China}
 \email{fli@nju.edu.cn}

\begin{abstract}
In this paper we investigate  the zero dielectric constant limit to the non-isentropic compressible
Euler-Maxwell system.  We justify this singular limit rigorously in the framework of
smooth solutions and obtain the non-isentropic compressible magnetohydrodynamic equations as the dielectric constant tends to zero.


 \end{abstract}



\maketitle

\renewcommand{\theequation}{\thesection.\arabic{equation}}
\setcounter{equation}{0}
\section{Introduction and Main Results} \label{S1}

The electromagnetic dynamics is governing by a  coupled PDE system describing the behavior of an electrically conducting fluid
and the electromagnetic fields.
In the absence of viscosity,  Hall effect, and heat conductivity, the system of electromagnetic dynamics can be written as (\cite{Im,EM})
\begin{align}
&\p_t \rho  +\dv(\rho\u)=0, \label{maa} \\
&\rho(\p_t\u +  \u\cdot \nabla\u)+\nabla p
  =\rho_\textrm{e}\i+\mu_0\j\times \H, \label{mab} \\
&\rho \theta(\p_tS+\u\cdot \nabla S)=(\j-\rho_\textrm{e}\u)\cdot (\i+\mu_0\u\times \H) ,\label{mac}\\
&\epsilon\p_t\i-\cu \H+\j=0, \label{mad}\\
&\p_t \H+\frac{1}{\mu_0 }\cu \i=0, \label{mae}\\
&\p_t (\rho_\textrm{e})+\dv \j=0,\label{maf}\\
&\epsilon\dv \i=\rho_\textrm{e},\quad \dv \H=0.\label{mag}
\end{align}
 Here  the unknowns $\rho,\u=(u_1,u_2,u_3)\in \mathbb{R}^3, S, \i=(E_1,E_2,E_3)\in \mathbb{R}^3,\H=(H_1,\linebreak H_2,H_3)\in \mathbb{R}^3$,
and $\rho_\textrm{e}$  denote the density, velocity,  entropy,    electric field,   magnetic field, and
  electric charge density, respectively.
The
current density $\j$ is expressed by Ohm's law, i.e.,
\begin{align}\label{ohmm}
\j-\rho_\textrm{e} \u =\sigma (\i+\mu_0\u\times \H).
\end{align}
The pressure $p$ and the  entropy $S$   satisfy the  Gibbs relation
\begin{equation}\label{gibbs}
\theta \mathrm{d}S=\mathrm{d}e +p\,\mathrm{d}\left(\frac{1}{\rho}\right),
\end{equation}
where $\theta$ and $e$ denote the temperature and  the internal energy of  the fluid.

To the authors' best knownledge, the only mathematical result on the system \eqref{maa}--\eqref{mag} was
obtained by  Kawashima~\cite{K} who established  the  global existence  of smooth solutions
in whole space $\mathbb{R}^2$ when the
  initial data are a small perturbation of some given constant state.
On the other hand, as it was pointed out in \cite{Im}, the
assumption that the electric charge density $\rho_\textrm{e}\simeq 0$ is
physically very reasonable for the study of plasmas. In this situation, we can eliminate
the terms involving $\rho_\textrm{e}$ in \eqref{maa}--\eqref{mag} and then obtain the following  non-isentropic compressible
Euler-Maxwell system:
\begin{align}
&\p_t \rho  +\dv(\rho\u)=0, \label{na} \\
&\rho(\p_t\u +  \u\cdot \nabla\u)+\nabla P
  = \mu_0\j\times \H, \label{nb} \\
&\rho \theta(\p_tS+\u\cdot \nabla S)=\j\cdot (\i+\mu_0\u\times \H) ,\label{nc}\\
&\epsilon\p_t\i-\cu \H+\j=0,\label{nd}\\
&\p_t \H+\frac{1}{\mu_0 }\cu \i=0, \quad \dv \H=0  \label{ne}
\end{align}
with
\begin{align}
   \quad \j =\sigma (\i+\mu_0\u\times \H). \label{oh}
\end{align}

Formally, if we take the  dielectric constant $\epsilon =0$ in
\eqref{nd}, i.e., the displacement current is negligible, then we
obtain $\j= \cu \H$. Thanks to \eqref{oh}, we can eliminate  the
electric field $\i$ in \eqref{nb},
\eqref{nc}   and \eqref{ne}, and finally obtain that
\begin{align}
&\p_t \rho  +\dv(\rho\u)=0, \label{nba} \\
&\rho(\p_t\u +  \u\cdot \nabla\u)+\nabla P
  =\mu_0  \cu \H\times \H, \label{nbb} \\
&\rho \theta(\p_tS+\u\cdot \nabla S) =\frac{1}{\sigma }|\cu \H|^2, \label{nbc}\\
&\partial_t \H -\cu(\u\times\H)= -\frac{1}{\sigma\mu_0}\cu (\cu\H ),\quad \dv\H=0.\label{nbd}
\end{align}
The equations \eqref{nba}--\eqref{nbd} is called  non-isentropic compressible magnetohydrodynamic equations with   infinite
Reynolds number which is used in describing some local processes in the cosmic system \cite{Hu87}.

The above formal derivation is usually referred as magnetohydrodynamic
approximation \cite{Im,EM}. In \cite{KS1,KS2}, Kawashima and
Shizuta justified this limit process rigorously to the complete magnetohydrodynamic
fluid system   in $\mathbb{R}^2$ for local and global small smooth solutions (small perturbations of some give constant state), respectively.
In \cite{JL},  we studied
the  magnetohydrodynamic
approximation  for the isentropic  electromagnetic fluid system in three-dimensional period domain  and obtained the
isentropic compressible magnetohydrodynamic
equations with explicit convergence rates. Recently, we extended the results in \cite{JL} to the  complete magnetohydrodynamic
fluid system   and obtained the  full
 compressible magnetohydrodynamic equations \cite{JL2}. We remark that the viscosities
  (including the shear and buck viscosities and heat conductivity coefficient) play  a crucial role in the proof process
 of \cite{JL2} and the inviscid case is left as an open problem there.

  The purpose of this paper is to solve this problem and give a  rigorous derivation of the
  compressible magnetohydrodynamic equations \eqref{nba}--\eqref{nbd} from the
  non-isentropic compressible Euler-Maxwell  system \eqref{na}--\eqref{oh} as the dielectric constant $\epsilon$ tends to
  zero. As in \cite{JL2}, we   consider the system \eqref{na}--\eqref{oh}
 in a periodic domain of $\mathbb{R}^3$, i.e., the torus
$\mathbb{T}^3=(\mathbb{R}/(2\pi \mathbb{Z}))^3$.

Below we take the  harmless physical constants
$\sigma$ and $\mu_0  $   to be one for simplicity of presentation.
For the system \eqref{na}--\eqref{oh}, it is more convenient to using the pressure $p$ instead of the density $\rho$ as an unknown.
 Thus  we reconsider the equations of state
as functions of $S$ and $p$, i.e., $\rho =r(S,p)$ and
$\theta=\Theta(S,p)$ for some positive smooth functions $r$ and
$\Theta$ defined for all $S$ and $p>0$, and satisfying
$\frac{\partial r(S,p)}{\partial p }>0$.  Moreover, in order to emphasize the unknowns depending  on the small parameter
$\epsilon$, we add the superscripts $\epsilon$ to the unknowns $ (p,\u,S, \i,\H)$ and rewrite the Euler-Maxwell  system  \eqref{na}--\eqref{oh} as
 \begin{align}
    & a(S^\ep ,p^\ep)(\partial_t p^\ep+\u^\ep\cdot \nabla p^\ep)+\dv \u^\ep=0,\label{nca}\\
& r(S^\ep,p^\ep)(\partial_t \u^\ep+\u^\ep\cdot \nabla \u^\ep)+\nabla p^\ep =(\i^\ep+\u^\ep\times \H^\ep)\times \H^\ep, \label{ncb}\\
 & b(S^\ep,p^\ep)(\partial_tS^\ep+\u^\ep\cdot \nabla S^\ep)=|\i^\ep+\u^\ep\times \H^\ep|^2,\label{ncc}\\
&  \epsilon\p_t\i^\ep-\cu \H^\ep + (\i^\ep+\u^\ep\times \H^\ep)=0\label{ncd}\\
&\p_t \H^\ep+\cu \i^\ep=0, \quad \dv \H^\ep=0.  \label{nce}
\end{align}
where $a(S^\ep,p^\ep)$ and  $b(S^\ep,p^\ep)$ are  defined as
\begin{align}\label{ncf}
   a(S^\epsilon,p^\epsilon)=\frac{1}{r(S^\epsilon,p^\epsilon)}\frac{\partial r(S^\epsilon,p^\epsilon)}{\partial p^\epsilon}, \quad
  b(S^\epsilon,p^\ep)= r(S^\ep,p^\ep)\Theta(S^\ep,p^\ep).
\end{align}
The system \eqref{nca}--\eqref{nce} are supplemented with initial data
\begin{align}\label{ncg}
 (p^\epsilon, \u^\epsilon, S^\epsilon, \i^\epsilon,\H^\epsilon)|_{t=0}
 =( p_0^\epsilon(x), \u_0^\epsilon(x),S_0^\epsilon(x), \i_0^\epsilon(x),\H_0^\epsilon(x)), \quad x\in \mathbb{T}^3.
\end{align}

We also rewrite the target equations \eqref{nba}--\eqref{nbd} (recall that $\mu_0\equiv\sigma\equiv 1$)  as
 \begin{align}
    & a(S^0 ,p^0)(\partial_t p^0+\u^0\cdot \nabla p^0)+\dv \u^0=0,\label{nda}\\
& r(S^0,p^0)(\partial_t \u^0+\u^0\cdot \nabla \u^0)+\nabla p^0 = \cu \H^0\times \H^0, \label{ndb}\\
 & b(S^0,p^0)(\partial_tS^0+\u^0\cdot \nabla S^0)=|\cu \H^0|^2,\label{ndc}\\
&  \p_t\H^0-\cu (\u^0\times \H^0)=- \cu\cu \H^0,\quad \dv \H^0=0.\label{ndd}
\end{align}
where $a(S^0,p^0)$ and $b(S^0,p^0)$ are  defined  through \eqref{ncf} with $(S^\epsilon,p^\epsilon)$ replaced by $(S^0,p^0)$.
The system  \eqref{nda}--\eqref{ndd} are  equipped with  initial data
\begin{align}\label{nde}
 (p^0, \u^0,S^0, \H^0)|_{t=0}
 =(p^0_0(x), \u^0_0(x), S_0^0(x), \H_0^0(x)), \quad x\in \mathbb{T}^3.
\end{align}

We remark that although the    electric field $\i^0$  does not appear in the system \eqref{nda}--\eqref{ndd}, it can be  induced  according to
the  relation
\begin{equation}\label{Ohm}
\i^0=\cu\H^0 - \u^0\times\H^0
\end{equation}
by the moving conductive flow in the magnetic field.

Before stating our main results, we recall the local existence  of smooth solutions to the problem \eqref{nda}--\eqref{nde}.
Since the system \eqref{nda}--\eqref{ndd} can be written as a symmetric hyperbolic-parabolic system, the results in \cite{VH} imply that

 \begin{prop} \label{Pa} Let $s> 7/2$ be an integer and
 assume that the initial data   $(p^0_0, \u^0_0,S^0_0,$ $ \H_0^0)$ satisfy
\begin{gather*}
 p^0_0, \u^0_0, S^0_0, \H_0^0\in H^{s+1}(\mathbb{T}^3), \ \ \dv \H^0_0 =0,\nonumber\\
    0<\bar p= \inf_{x\in \mathbb{T}^3}p^0_0(x)\leq p^0_0(x)\leq
    \bar {\bar p}= \sup_{x\in \mathbb{T}^3}p^0_0(x)<+\infty,\\
    0<\bar S= \inf_{x\in \mathbb{T}^3}S^0_0(x)\leq S^0_0(x)\leq
    \bar {\bar S}= \sup_{x\in \mathbb{T}^3}S^0_0(x)<+\infty
\end{gather*}
 for some positive constants $\bar p,\, \bar{\bar p},\, \bar S$, and $\bar{\bar S}$. Then there exist positive
 constants $T_*$\, $($the maximal time interval, $ 0<T_*\leq +\infty )$ and $\hat p, \tilde{p}, \hat S, \tilde{S} $, such that the problem
\eqref{nda}--\eqref{nde} has a unique classical solution $(p^0,\u^0,S^0,\H^0)$ satisfying $\dv  \H^0=0$ and
\begin{gather*}
   p^0, \u^0, S^0 \in  C^l([0,T_*),H^{s+1-l}(\mathbb{T}^3)), \  \H^0 \in  C^l([0,T_*),H^{s+1-2l}(\mathbb{T}^3)), \ \     l=0,1;\ \
 \\
    0<\hat p= \inf_{(x,t)\in \mathbb{T}^3 \times [0,T_*)}p^0(x,t)\leq  p^0(x,t)\leq
    {\tilde p}= \sup_{(x,t)\in \mathbb{T}^3 \times [0,T_*)}p^0(x,t)<+\infty,\\
       0<\hat S= \inf_{(x,t)\in \mathbb{T}^3 \times [0,T_*)}S^0(x,t)\leq S^0(x,t)\leq
    {\tilde S}= \sup_{(x,t)\in \mathbb{T}^3 \times [0,T_*)}S^0(x,t)<+ \infty.
\end{gather*}

 \end{prop}

The  main result of this paper can be stated as follows.
\begin{thm}\label{th}
Let $s>7/2$ be an integer and  $(p^0, \u^0,S^0, \H^0)$ the unique classical solution to the problem
\eqref{nda}--\eqref{nde} given in
Proposition \ref{Pa}.
 Suppose
that the initial data  $(p^\epsilon_0, \u^\epsilon_0,S_0^\epsilon, \i_0^\epsilon,
\H_0^\epsilon)$ satisfy
$$
 p^\epsilon_0, \u^\epsilon_0,S_0^\epsilon,  \i^\epsilon_0, \H^\epsilon_0\in H^{s}(\mathbb{T}^3),
 \  \inf_{x\in \mathbb{T}^3}
   p^\epsilon_0(x)>0, \      \inf_{x\in \mathbb{T}^3}
  S^\epsilon_0(x)>0, \  \dv \H^\epsilon_0 =0,
 $$
and
\begin{align}
&  \Vert (p^\epsilon_0-p^0_0, \u^\epsilon_0-\u^0_0, S^\epsilon_0-S^0_0,
\H_0^\epsilon-\H_0^0) \Vert_{H^s(\mathbb{T}^3)}\nonumber\\
&  \qquad \qquad\qquad\quad
+ \sqrt{\epsilon} \left\Vert \i^\epsilon_0- ( \cu\H^0_0
- \u^0_0\times\H^0_0 ) \right\Vert_{H^s(\mathbb{T}^3)} \leq  L_0 {\epsilon} \label{ivda}
 \end{align}
for some constant $L_0>0$. Then, for any $T_0\in (0,T_* )$,  there exist
   a constant  $L>0$, and
a sufficient small constant $\epsilon_0>0$ such that, for any $\epsilon\in
(0,\epsilon_0]$,  the problem \eqref{nca}--\eqref{ncg} has a unique smooth solution $(p^\epsilon,
\u^\epsilon, S^\epsilon, \i^\epsilon,\H^\epsilon)$ on $[0,T_0]$ enjoying
\begin{align}\label{iivda}
  &  \Vert (p^\epsilon-p^0, \u^\epsilon-\u^0, S^\epsilon-S^0,\H^\epsilon-\H^0)(t)
\Vert_{H^s(\mathbb{T}^3)} \nonumber\\
& \qquad \quad +  \sqrt{\epsilon}\left\Vert\left\{\i^\epsilon- ( \cu\H^0
- \u^0\times\H^0 )\right\}(t)\right\Vert_{H^s(\mathbb{T}^3)} \leq L {\epsilon},  \ \ t\in [0,T_0].
 \end{align}
\end{thm}

We shall prove   Theorem \ref{th} by  adapting the  elaborate nonlinear energy method  inspired by
\cite{JL,JL2}.  The key point of the proof is to derive the error system (see \eqref{error1}--\eqref{error4} below) and obtain
the uniform estimates in a fixed time interval independent of   $\epsilon$.
 As mentioned before,  the zero dielectric constant limit to the complete magnetohydrodynamic
fluid system were studied in \cite{JL2} where the viscosity and heat conductivity terms
in the complete electromagnetic fluid system play a crucial role in the derivation of the uniform estimates.
In our case, all diffusion terms disappear and we shall make full use of the special structural of the system \eqref{nca}--\eqref{ncd}
 to obtain the desired uniformly estimates. A direct but crucial observation is that there is a damping term
$ \i^\epsilon-\i^0 $  in the electric field equations which
control the terms involving $\i^\epsilon-\i^0$ in the momentum equations, entropy equation,
and electric filed equations. In order to obtain the desired higher order estimates to the error system, we shall
also  modify some ideas developed in \cite{MS01,JJL4} which is quite different to the   isentropic
case \cite{JL} and the viscous non-isentropic case \cite{JL2}.

\begin{rem}
  The inequality \eqref{iivda} implies that the sequences $(p^\epsilon, \u^\epsilon,S^\epsilon, \H^\epsilon)$
  converge strongly to $(p^0,\u^0,S^0,\H^0)$ in $L^\infty(0,T; H^{s}(\mathbb{T}^3))$ and
  $\i^\epsilon$ converges strongly to $\i^0$ in $L^\infty(0,T; H^{s}(\mathbb{T}^3))$   but with different convergence rates, where
$\i^0$ is defined by \eqref{Ohm}.
\end{rem}

\begin{rem}
  For the local existence of solutions $(p^0,\u^0,S^0,\H^0)$ to the problem \eqref{nda}--\eqref{nde}, the
  assumption on the regularity of initial data $(p^0_0,\u^0_0,\theta^0_0,\H^0_0)$ belongs to
  $H^s(\mathbb{T}^3)$, $s>7/2$, is enough. Here we have added more regularity assumption in Proposition \ref{Pa} to
  obtain more regular solutions which are needed in the proof of Theorem \ref{th}.
  The higher regularity assumption on the target equations can provide
  a simpler arguments in this paper. To investigate the singular limit, another way is to obtain higher order
  uniform estimates directly, see, for example, \cite{MS01} on zero Mach number limit
  to non-isentropic Euler equations.
\end{rem}

\begin{rem}
In this paper we just consider the  periodic domain case, it is more interesting to study the same problem in a spatial domain with boundary
which will be our future study. We remark that in this case the boundary must be analyzed very carefully,
the interested reader can refer \cite{A,B1,B2,B,S,Sc}, and  among others on the zero Mach number limit of the compressible Euler equations,
and \cite{Rub} on singular limits of zero Alfv\'{e}n number for the equations of magneto-fluid dynamics.
In \cite{B1,B2}, some new pioneering ideas are introduced, which can be applied to the convergence study of
the singular limit in the data space, see \cite{B1,Rub,B} for the details.
\end{rem}


\begin{rem}
 We point out that the zero  dielectric constant limit is a singular limit
 and similar to the  zero Mach number limit in some sense, see \cite{A,B,JJL3,JJL4,JJLX,KM1,MS01,S,Sc} and
 the references cited therein.
\end{rem}

\begin{rem}
  It is obvious that if we let $\sigma\rightarrow \infty$ in \eqref{nba}--\eqref{nbd}, we will obtain formally
  the well-known ideal non-isentropic magnetohydrodynamic equations. It is interesting to
  establish this limit rigorously.
\end{rem}

\medskip
Before ending this introduction, we give some notations and recall some basic facts which
will be frequently used throughout this paper.

(1) We denote by $\langle \cdot,\cdot\rangle$ the standard inner product in $L^2(\mathbb{T}^3)$
with $\langle f,f\rangle=\|f\|^2$, by
$H^k$ the standard Sobolev space $W^{k,2}$ with $\|\cdot\|_{k}$ being
the corresponding norm ($\|\cdot\|_{0}\equiv\|\cdot\|)$.  The notation $\|(A_1,A_2, \dots,
 A_k)\|$ means the summation of $\|A_i\|,i=1,\dots,k$,
 and it also applies to  other norms.
For the multi-index $\alpha = (\alpha_1,  \alpha_2, \alpha_3)$,  we
denote  $\partial_x^\alpha =\partial^{\alpha_1}_{x_1}\partial^{\alpha_2}_{x_2}
\partial^{\alpha_3}_{x_3}$ and
$|\alpha|=|\alpha_1|+|\alpha_2|+|\alpha_3|$. For the integer $l$, the symbol $D^l_x$ denotes
the summation of all terms $\partial_x^\alpha$ with the multi-index $\alpha$ satisfying $|\alpha|=l$. We use $C_i$,
$\delta_i$, $K_i$, and $K$ to denote the constants which are independent of
$\epsilon$ and may change from line to line. We also omit the  spatial domain $\mathbb{T}^3$
in integrals for convenience.

(2) We shall frequently use the following Moser-type calculus
inequalities (see \cite{KM1}):

\hskip 4mm (i)\ \ For $f,g\in H^s(\mathbb{T}^3)\cap L^\infty(\mathbb{T}^3)$ and $|\alpha|\leq
s$, $s>3/2$, it holds that
\begin{align}\label{ma}
\|\partial^\alpha_x(fg)\| \leq C_s(\|f\|_{L^\infty}\|D^s_x
g\| +\|g\|_{L^\infty}\|D^s_x f\|).
\end{align}

\hskip 4mm (ii)\ \ For $f\in H^s(\mathbb{T}^3), D_x^1 f\in L^\infty(\mathbb{T}^3), g\in H^{s-1}(\mathbb{T}^3)\cap
L^\infty(\mathbb{T}^3)$ and $|\alpha|\leq s$, $s>5/2$, it holds that
\begin{align}\label{mb}
\quad  \ \ \|\partial^\alpha_x(fg)-f \partial^\alpha_xg\|\leq
C_s(\|D^1_x f\|_{L^\infty}\|D^{s-1}_x g\| +\|g\|_{L^\infty}\|D^s_xf\|). 
\end{align}

(3) Let $s> 3/2$, $f\in C^s(\mathbb{T}^3)$, and  $u\in H^s(\mathbb{T}^3)$, then for each multi-index $\alpha$, $1\leq |\alpha| \leq s$, we have
(\cite{Mo,KM1}):
\begin{align}\label{mo}
   \|\partial^\alpha_x (f(u))\| \leq C(1+\|u\|_{L^\infty}^{|\alpha|-1})\|u\|_{|\alpha|};
\end{align}
moreover, if $f(0)=0$, then (\cite{Ho97})
\begin{align}\label{ho}
  \|\partial^\alpha_x(f(u))\|\leq C( \|u\|_s)\|u\|_s.
\end{align}

This paper is organized as follows. In Section \ref{S2}, we utilize the primitive system \eqref{nca}--\eqref{nce} and the
target system \eqref{nda}--\eqref{ndd} to  derive an  error
system and state the local existence of  solutions to the error system.
  In Section \ref{S3} we give  \emph{a priori} energy
estimates to the error system  and present  the proof of Theorem \ref{th}.

\renewcommand{\theequation}{\thesection.\arabic{equation}}
\setcounter{equation}{0}
\section{Derivation of an error system and  local existence} \label{S2}

In this section  we first derive an error system from the original
system \eqref{nca}--\eqref{nce} and the target equations
\eqref{nda}--\eqref{ndd}. Then we state the local existence of  smooth solutions to
this error system.

Setting $P^\epsilon=p^\epsilon- p^0,   \U^\epsilon=\u^\epsilon-\u^0, \Phi^\epsilon= S^\epsilon- S^0,
\F^\epsilon=\i^\epsilon-\i^0,  \G^\epsilon=\H^\epsilon-\H^0$ and utilizing
      the system
\eqref{nca}--\eqref{nce} and the system \eqref{nda}--\eqref{ndd} with \eqref{Ohm}, we
  obtain that
 \begin{align}
 &  a(\Phi^\epsilon+S^0, P^\epsilon+p^0) \{\partial_t  P^\epsilon  +(\U^\epsilon+ \u^0)\cdot \nabla P^\epsilon\}+\dv \U^\epsilon =f_1^\epsilon, \label{error1} \\
 & r(\Phi^\epsilon+S^0, P^\epsilon+p^0) \{\partial_t\U^\epsilon  +(\U^\epsilon+\u^0)\cdot \nabla\U^\epsilon\}
 +\nabla \Phi^\epsilon=\mathbf{f}_2^\epsilon, \label{error2}\\
 & b(\Phi^\epsilon+S^0, P^\epsilon+p^0) \{\partial_t\Phi^\epsilon  +(\U^\epsilon+\u^0)\cdot \nabla\Phi^\epsilon\}=\mathbf{f}_3^\epsilon,      \label{error22}\\
  & \epsilon \partial_t \F^\epsilon - \cu \G^\epsilon
 =\mathbf{f}_4^\epsilon,  \label{error3}\\
 & \partial_t \G^\epsilon+\cu \F^\epsilon =0,\quad  \dv \G^\epsilon=0,  \label{error4}
 \end{align}
where $f_1^\epsilon$, $\mathbf{f}_2^\epsilon$, $\mathbf{f}_3^\epsilon$,  and $\mathbf{f}_4^\epsilon$ are defined as follows:
\begin{align*}
  f_1^\epsilon =& -[a(\Phi^\epsilon+S^0, P^\epsilon+p^0)-a(S^0,p^0)][\partial_t p^0 + \u^0\cdot\nabla p^0]\nonumber\\
  &   -a(\Phi^\epsilon+S^0, P^\epsilon+p^0)(\U^\epsilon\cdot \nabla p^0),   \\
 \mathbf{f}_2^\epsilon =& -[r(\Phi^\epsilon+S^0, P^\epsilon+p^0)-r(S^0,p^0)][\partial_t \u^0 +\u^0\cdot\nabla\u^0]\nonumber\\
  &  -r(\Phi^\epsilon+S^0, P^\epsilon+p^0)(\U^\epsilon\cdot \nabla\u^0)\nonumber\\
 &               - \cu \H^0\times \H^0
              +[\F^\epsilon+\u^0\times \G^\epsilon+\U^\epsilon\times \H^0]\times \H^0 \nonumber\\
  &  +[\F^\epsilon+\u^0\times \G^\epsilon+\U^\epsilon\times\H^0]\times \G^\epsilon+ (\U^\epsilon\times \G^\epsilon)\times (\G^\epsilon+\H^0),\\
 \mathbf{f}_3^\epsilon =&-[b(\Phi^\epsilon+S^0, P^\epsilon+p^0)-b(S^0,p^0)][\partial_t S^0 +\u^0\cdot\nabla S^0]\nonumber\\
  &   -b(\Phi^\epsilon+S^0, P^\epsilon+p^0)(\U^\epsilon\cdot \nabla S^0) \nonumber\\
   &  + |\F^\epsilon+\U^\epsilon\times \G^\epsilon|^2+ |\u^0\times \G^\epsilon+\U^\epsilon\times \H^0|^2\nonumber\\
    &  +  {2} (\F^\epsilon+\U^\epsilon\times \G^\epsilon)\cdot
                   [\cu \H^0+\u^0\times \G^\epsilon+\U^\epsilon\times \H^0]\nonumber\\
      &  +    {2} \cu \H^0\cdot (\u^0\times \G^\epsilon+\U^\epsilon\times \H^0),\\
     \mathbf{f}_4^\epsilon= &-  [\F^\epsilon+\U^\epsilon\times \H^0+\u^0\times \G^\epsilon]-  \U^\epsilon\times \G^\epsilon\nonumber\\
 & - {\epsilon} \partial_t \cu \H^0+\epsilon\partial_t(\u^0\times \H^0).
\end{align*}
The system \eqref{error1}--\eqref{error4} are supplemented  with  initial data
 \begin{align}\label{error5}
 &  (  P^\epsilon,\U^\epsilon,\Phi^\epsilon,\F^\epsilon,\G^\epsilon)|_{t=0}=
 (  P^\epsilon_0,\U^\epsilon_0,\Phi^\epsilon_0,\F^\epsilon_0,\G^\epsilon_0)\nonumber\\
  &  \qquad   := \big( p^\epsilon_0- p^0_0, \u_0^\epsilon-\u^0_0,  S^\epsilon_0- S^0_0,
   \i^\epsilon_0- ( \cu\H^0_0 - \u^0_0\times\H^0_0 ),\H^\epsilon_0-\H^0_0\big).
 \end{align}

 Denote
 \begin{align*}
 &\W^\epsilon=\left(\begin{array}{c}
                    P^\epsilon \\
                    \U^\epsilon\\
                    \Phi^\epsilon\\
                   \F^\epsilon \\
                    \G^\epsilon
                  \end{array}\right),
                  \ \
       \W^\epsilon_0=\left(\begin{array}{c}
                      P^\epsilon_0 \\
                    \U^\epsilon_0\\
                     \Phi^\epsilon_0\\
                     \F^\epsilon_0\\
                     \G^\epsilon_0\\
                  \end{array}\right), \ \  \s^\epsilon(\W^\epsilon)=\left(\begin{array}{c}
                   f^\epsilon_1\\
                     \mathbf{f}^\epsilon_2\\
                    \mathbf{f}^\epsilon_3\\
                    \mathbf{f}^\epsilon_4\\
                    \mathbf{0}
                    \end{array}
                    \right),\\
  & \D^\epsilon=\left(\begin{array}{cc}
                   \D^\epsilon_1 & \mathbf{0} \\
                   \mathbf{0} & \left(\begin{array}{cc}
                    \epsilon   \mathbf{I}_{3} & \mathbf{0}\\
                    \mathbf{0} &  \mathbf{I}_{3}
                    \end{array}
                    \right)
                  \end{array}\right), \\
 & \D^\epsilon_1=   \left(\begin{array}{ccc}
                    a(\Phi^\epsilon+S^0, P^\epsilon+p^0) & 0 & 0  \\
                    0 &  r(\Phi^\epsilon+S^0, P^\epsilon+p^0) \mathbf{I}_{3} & \mathbf{0}\\
                    0 & \mathbf{0} &  b(\Phi^\epsilon+S^0, P^\epsilon+p^0)
                    \end{array}\right),\\
 &   \A^\epsilon_i=\left(\begin{array}{cc}
                    \left(\begin{array}{ccc}
                    (\U^\epsilon+\u^0)_i & e_i & 0  \\
                    e^\mathrm{T}_i &  (\U^\epsilon+\u^0)_i \mathbf{I}_{3} & \mathbf{0}\\
                    0 & \mathbf{0} &  (\U^\epsilon+\u^0)_i
                    \end{array}\right)     & \mathbf{0} \\
                    \mathbf{0} &  \left(\begin{array}{cc}
                      \mathbf{0}& B_{i} \\
                    B_{i}^\mathrm{T}  & \mathbf{0}
                    \end{array}
                    \right)
                  \end{array}\right),
 \end{align*}
where
 $(e_1, e_2, e_3)$ is the canonical basis of $\mathbb{R}^3$, $\mathbf{I}_{d}$ ($d = 3,5$)
 is the $d\times d$ unit matrix, $y_i$ denotes the $i$-th component of $y\in \mathbb{ R}^3$, and
\begin{align*}
B_1 =\left(\begin{array}{ccc}
 0 & 0 & 0 \\
0 & 0 &  1\\
0 & -1 &  0
\end{array}
\right), \quad
 B_2  =\left(\begin{array}{ccc}
 0 & 0 & -1 \\
0 & 0 &  0\\
1 & 0 &  0
\end{array}
\right), \quad
B_3 =\left(\begin{array}{ccc}
 0 & 1 & 0 \\
-1 & 0 &  0\\
0 & 0 &  0
\end{array}
\right).
\end{align*}

Using these notations we can rewrite the  problem
\eqref{error1}--\eqref{error5} as
\begin{align}\label{error6}
  \left\{\begin{aligned}
&  \D^\epsilon \partial_t \W^\epsilon +\sum^{3}_{i=1}\A^\epsilon_i \W^\epsilon_{x_i}
  =\s^\epsilon(\W^\epsilon),\\
 & \W^\epsilon|_{t=0}= \W^\epsilon_0.
 \end{aligned} \right.
\end{align}
Obviously, the system  in
\eqref{error6}  is  a quasilinear symmetric
hyperbolic one. Thus, we can apply the result  of Majda \cite{M84}    to obtain   the following local existence of
smooth solutions to the problem \eqref{error6}.

\begin{prop} \label{Pb}
Let  $s>7/2 $ be an integer and $( p^0_0, \u^0_0,  S^0_0,  \H_0^0)$ satisfy the conditions in Proposition \ref{Pa}.
 Assume that the initial data $( P^\epsilon_0, \U^\epsilon_0, \Phi^\epsilon_0, \F_0^\epsilon, \G_0^\epsilon)$ satisfy
\begin{gather*}
   P^\epsilon_0, \U^\epsilon_0,\Phi^\epsilon_0, \F^\epsilon_0,  \G^\epsilon_0\in
  H^s(\mathbb{T}^3), \quad \dv \G^\epsilon_0 =0,\\
\inf_{x\in \mathbb{T}^3} P^\epsilon_0(x)>0,\quad   \inf_{x\in \mathbb{T}^3}\Phi^\epsilon_0(x)>0,\quad
\|\Phi^\epsilon_0\|_s\leq \delta,\quad
\| P^\epsilon_0\|_s\leq \delta
\end{gather*}
  for some small constant $\delta>0$.
Then there exist positive constants $T^\epsilon\,(0<T^\epsilon\leq +\infty)$ and $K$,  such that the problem
\eqref{error6} has a unique classical solution $( P^\epsilon,
\U^\epsilon, \Phi^\epsilon, \linebreak \F^\epsilon, \G^\epsilon)$ satisfying
\begin{gather*}
    P^\epsilon,
\U^\epsilon, \Phi^\epsilon, \F^\epsilon, \G^\epsilon  \in  C^l([0,T^\epsilon),H^{s-l}(\mathbb{T}^3)),\  l=0,1; \, \,  \, \,  \dv \G^\epsilon =0; \\
     \| P^\epsilon(t)\|_{s}\leq
   K\delta , \quad  \|\Phi^\epsilon(t)\|_{s}\leq
   K\delta, \ \ \  t\in [0,T^\epsilon).
  \end{gather*}

 \end{prop}

 Notice that for smooth solutions, the non-isentropic Euler-Maxwell
  system \eqref{nca}--\eqref{nce} with initial data \eqref{ncg}  is  equivalent to
\eqref{error1}--\eqref{error5} or \eqref{error6} on $[0,T]$. Therefore, in order to obtain  the convergence of
 the Euler-Maxwell system \eqref{nca}--\eqref{nce} to the
 compressible magnetohydrodynamic equations \eqref{nda}--\eqref{ndd},
  we  need to establish the uniform decay estimates in some time interval $[0,T]$ with respect
  to the parameter $\epsilon$ of the solution to the error system \eqref{error1}--\eqref{error5}.
We shall   present these estimates in the next section.

\renewcommand{\theequation}{\thesection.\arabic{equation}}
\setcounter{equation}{0}
\section{Uniform energy estimates and proof of Theorem \ref{th}} \label{S3}

 In this section  we shall derive the uniform decay estimates with respect to the parameter
  $\epsilon$ of the solution to the problem \eqref{error1}--\eqref{error5} and
  justify rigorously the convergence of the non-isentropic Euler-Maxwell system \eqref{nca}--\eqref{nce} to the
   compressible magnetohydrodynamic equations \eqref{nda}--\eqref{ndd}.
Here we shall make full use of the structure of the system \eqref{error1}--\eqref{error4} and
Proposition \ref{Pb}, and  adapt  some techniques developed in \cite{JL,JL2,JJL4,MS01}.

We first establish the convergence rate of the error system \eqref{error1}--\eqref{error3} by obtaining the \emph{a priori} estimates
uniformly in $\epsilon$. For simplicity of presentation, we define
\begin{align*}
 &\|\mathcal{E}^\epsilon(t)\|^2_s\    = \|( P^\epsilon,\U^\epsilon,\Phi^\epsilon, \G^\epsilon)(t)\|^2_{s},\\
 &\v \mathcal{E}^\epsilon(t)\v ^2_s =\|\mathcal{E}^\epsilon(t)\| ^2_s+ \epsilon \Vert \F^\epsilon \Vert^2_{s},\\
&\v\mathcal{E}^\epsilon\v_{s,T}\  =\sup_{0<t<T}\v\mathcal{E}^\epsilon(t)\v_s.
\end{align*}

The crucial estimate of our paper is the following uniform-in-$\epsilon$ result on the error system
\eqref{error1}--\eqref{error4}.

\begin{prop}\label{P31}
   Let $s>7/2$ be an integer and assume that the initial data  $(  P^\epsilon_0,\U^\epsilon_0,\Phi^\epsilon_0, \F^\epsilon_0,\G^\epsilon_0)$ satisfy
\begin{align}\label{ww}
\|( P^\epsilon_0,\U^\epsilon_0,\Phi^\epsilon_0, \G^\epsilon_0) \|^2_{s}+ \epsilon\Vert
   \F^\epsilon_0 \Vert^2_{s}  =\v \mathcal{E}^\epsilon(t=0)\v _{s}^2\leq M_0{\epsilon}^2
 \end{align}
for sufficiently small $\epsilon$ and   some constant $M_0>0$   independent of $\epsilon$.
Then, for any $T_0\in (0, T_*)$, there are
  two constants $ M_1 > 0$  and $\epsilon_1 > 0$  depending only on $T_0$, such that
for all $\epsilon\in (0,\epsilon_1]$, it holds that $T^\epsilon\geq T_0$
and the solution $(  P^\epsilon, \U^\epsilon,\Phi^\epsilon, \F^\epsilon, \G^\epsilon)$ of the problem
\eqref{error1}--\eqref{error5}, well-defined in $[0, T_0]$, enjoys
 \begin{align}\label{www}
   \v \mathcal{E}^\epsilon\v _{s,T_0} \leq M_1 {\epsilon}.
 \end{align}
   \end{prop}

In order to prove Proposition \ref{P31}  we first derive the  following a priori
estimates on $[0,T]$ with $T \equiv  T_\epsilon = \min\{ T_1, T^\epsilon \}$ for some given  $\hat T<1$ and any $T_1<\hat T$ independent of $\epsilon$.

\medskip
\subsection{$L^2$ estimates}
\begin{lem}\label{La}
   Under the assumptions in Proposition \ref{P31}, it holds that for  all $0<t<T $ and sufficiently small $\epsilon$,
   \begin{align}\label{L2}
     &
     \|( P^\epsilon,\U^\epsilon,\Phi^\epsilon, \G^\epsilon)(t)\|^2+\epsilon\|  \F^\epsilon(t)\| ^2 +\frac32\int^t_0 \|\F^\epsilon(\tau)\|^2\dif \tau\nonumber\\
     \leq  &
     C\Big\{ \|( P^\epsilon,\U^\epsilon,\Phi^\epsilon, \G^\epsilon)(t)\|^2+\epsilon\|  \F^\epsilon(t)\| ^2\Big\}(t=0)+ C_T\epsilon^2
\nonumber\\
    &   +\int^t_0\big\{\eta_2\|\F^\epsilon(\tau)\|^2+\eta_3\|\F^\epsilon(\tau)\|^4+\big[ \eta_1\|\F^\epsilon(\tau)\|^2_2 +C(1+\|\mathcal{E}^\epsilon(\tau)\|^2_s   \nonumber\\
&+\|\mathcal{E}^\epsilon(\tau)\|^4_s+\|\mathcal{E}^\epsilon(\tau)\|^8_s)\big]
 \| (P^\epsilon,\U^\epsilon,\Phi^\epsilon,\G^\epsilon )(\tau)\|^2\big\}\dif \tau,
       \end{align}
       where $\eta_1$, $\eta_2$, and $\eta_3$ are sufficiently small positive constants.
\end{lem}

\begin{proof}
Multiplying \eqref{error1} by $P^\epsilon$, \eqref{error2} by $\U^\epsilon$, \eqref{error22} by $\Phi^\epsilon$,
 and integrating them over $\mathbb{T}^3$ respectively,  we obtain that
\begin{align}\label{LU2}
  &\langle a(\Phi^\epsilon+S^0, P^\epsilon+p^0) \partial_t P^\epsilon,P^\epsilon\rangle
  +\langle r(\Phi^\epsilon+S^0, P^\epsilon+p^0) \partial_t \U^\epsilon, \U^\epsilon\rangle\nonumber\\
  &  +\langle b(\Phi^\epsilon+S^0, P^\epsilon+p^0)\partial_t \Phi^\epsilon,\Phi^\epsilon\rangle
   \nonumber\\
  = 
  & -  \left \langle a(\Phi^\epsilon+S^0, P^\epsilon+p^0)(\U^\epsilon +\u^0) \cdot \nabla P^\epsilon , P^\epsilon\right\rangle\nonumber\\
 & -  \left \langle r(\Phi^\epsilon+S^0, P^\epsilon+p^0)(\U^\epsilon +\u^0) \cdot \nabla \U^\epsilon , \U^\epsilon\right\rangle \nonumber\\
   & -  \left \langle b(\Phi^\epsilon+S^0, P^\epsilon+p^0)(\U^\epsilon +\u^0) \cdot \nabla \Phi^\epsilon , \Phi^\epsilon\right\rangle\nonumber\\
  & +    \left \langle f^\epsilon_1,P^\epsilon\right\rangle +   \left \langle \mathbf{f}^\epsilon_2,\U^\epsilon\right\rangle
   +    \left \langle \mathbf{f}^\epsilon_3,\Phi^\epsilon\right\rangle .
   \end{align}

Thanks to the positivity and smoothness of $a(\Phi^\epsilon+S^0, P^\epsilon+p^0)$, $r(\Phi^\epsilon+S^0, P^\epsilon+p^0)$  and $b(\Phi^\epsilon+S^0, P^\epsilon+p^0)$,
  Proposition \ref{Pb}, the regularity of $( p^0,\u^0,S^0,\H^0)$,
 Cauchy-Schwarz's inequality, Sobolev's imbedding,  and \eqref{mo},
we get directly from \eqref{error1}, \eqref{error2}, and \eqref{error22}
 that
\begin{align}
&\|(\partial_{t}a(\Phi^\epsilon+S^0, P^\epsilon+p^0), \partial_{t}r(\Phi^\epsilon+S^0, P^\epsilon+p^0),\partial_{t}b(\Phi^\epsilon+S^0, P^\epsilon+p^0) )\|_{L^\infty}\nonumber\\
\leq&
 \|(\partial_{t}a(\Phi^\epsilon+S^0, P^\epsilon+p^0),\partial_{t}r(\Phi^\epsilon+S^0, P^\epsilon+p^0),\partial_{t}b(\Phi^\epsilon+S^0, P^\epsilon+p^0))\|_{2}\nonumber\\
 \leq &
 \eta_1\|\F^\epsilon\|^2_2+ C(\|\mathcal{E}^\epsilon(t)\|^4_s+\|\mathcal{E}^\epsilon(t)\|^2_s+1)  \label{wb}
\end{align}
for any $\eta_1>0$ and
\begin{align}
& \|(\nabla a(\Phi^\epsilon+S^0, P^\epsilon+p^0), \nabla r(\Phi^\epsilon+S^0, P^\epsilon+p^0),\nabla b(\Phi^\epsilon+S^0, P^\epsilon+p^0))\|_{L^\infty}\nonumber\\
\leq & C (1+\|\mathcal{E}^\epsilon(t)\|_s+\|\mathcal{E}^\epsilon(t)\|^2_s). \label{wc}
\end{align}
Thus, the first three terms
on the right-hand side of \eqref{LU2} can be estimated as follows:
\begin{align}
  &  \left| \langle a(\Phi^\epsilon+S^0, P^\epsilon+p^0)(\U^\epsilon +\u^0) \cdot \nabla P^\epsilon , P^\epsilon\rangle\right|\nonumber\\
 & + \left| \langle r(\Phi^\epsilon+S^0, P^\epsilon+p^0)(\U^\epsilon +\u^0) \cdot \nabla \U^\epsilon , \U^\epsilon\rangle\right|\nonumber\\
  &+ \left| \langle b(\Phi^\epsilon+S^0, P^\epsilon+p^0)(\U^\epsilon +\u^0) \cdot \nabla \Phi^\epsilon , \Phi^\epsilon\rangle\right|\nonumber\\
  \leq  & C (1+\|\mathcal{E}^\epsilon(t)\|_s+\|\mathcal{E}^\epsilon(t)\|^2_s+\|\mathcal{E}^\epsilon(t)\|^4_s)
  (\| P^\epsilon\|^2+\|\U^\epsilon\|^2+ \|\Phi^\epsilon\|^2). \label{wcc}
  \end{align}
By the definition of $f_1^\epsilon, \mathbf{f}^\epsilon_2$ and $\mathbf{f}^\epsilon_3$, the   regularity of $( p^0,\u^0,S^0,\H^0)$, and
 Cauchy-Schwarz's inequality, we have
 \begin{align}
   & \left \langle f^\epsilon_1,P^\epsilon\right\rangle +   \left \langle \mathbf{f}^\epsilon_2,\U^\epsilon\right\rangle
   +    \left \langle \mathbf{f}^\epsilon_3,\Phi^\epsilon\right\rangle\nonumber\\
   \leq &  C\epsilon^2 +\eta_2\|\F^\epsilon\|^2+\eta_3\|\F^\epsilon\|^4\nonumber\\
  & +C(1+\|\mathcal{E}^\epsilon(t)\|^2_s+\|\mathcal{E}^\epsilon(t)\|^4_s+\|\mathcal{E}^\epsilon(t)\|^8_s)
  (\| P^\epsilon\|^2+\|\U^\epsilon\|^2+ \|\Phi^\epsilon\|^2) \label{wd}
 \end{align}
for any $\eta_2>0$ and  $\eta_3>0$.

Putting \eqref{wcc} and \eqref{wd} into \eqref{LU2} and noticing \eqref{wb}, we  arrive at
\begin{align}
& \langle a(\Phi^\epsilon+S^0, P^\epsilon+p^0)  P^\epsilon,  P^\epsilon\rangle+
\langle r(\Phi^\epsilon+S^0, P^\epsilon+p^0)  \U^\epsilon,  \U^\epsilon\rangle\nonumber\\
&+ \langle b(\Phi^\epsilon+S^0, P^\epsilon+p^0)  \Phi^\epsilon,  \Phi^\epsilon\rangle\nonumber\\
\leq & \big\{ \langle a(\Phi^\epsilon+S^0, P^\epsilon+p^0)  P^\epsilon,  P^\epsilon\rangle+
\langle r(\Phi^\epsilon+S^0, P^\epsilon+p^0)  \U^\epsilon,  U^\epsilon\rangle\nonumber\\
&+ \langle b(\Phi^\epsilon+S^0, P^\epsilon+p^0)  \Phi^\epsilon,  \Phi^\epsilon\rangle\big\}\big|_{t=0}  +\int^t_0\Big\{ C\epsilon^2 +\eta_2\|\F^\epsilon\|^2+\eta_3\|\F^\epsilon\|^4    \nonumber\\
& +\big[ \eta_1\|\F^\epsilon\|^2_2+C(1+\|\mathcal{E}^\epsilon\|^2_s+\|\mathcal{E}^\epsilon\|^4_s+\|\mathcal{E}^\epsilon\|^8_s)\big]
 \| (P^\epsilon,\U^\epsilon,\Phi^\epsilon)\|^2\Big\}(\tau)\dif \tau. \label{we}
 \end{align}
Moreover, we have
\begin{align}  \label{wf}
&  \|P^\epsilon\|^2+\|\U^\epsilon\|^2+\|\Phi^\epsilon\|^2\nonumber\\
\leq&
\|(a(\Phi^\epsilon+S^0, P^\epsilon+p^0))^{-1}\|_{L^\infty}\langle a (\Phi^\epsilon+S^0, P^\epsilon+p^0)P^\epsilon,
P^\epsilon\rangle\nonumber\\
&
   + \|(r(\Phi^\epsilon+S^0, P^\epsilon+p^0))^{-1}\|_{L^\infty}\langle r(\Phi^\epsilon+S^0, P^\epsilon+p^0)  \U^\epsilon, \U^\epsilon\rangle\nonumber\\
   &+ \|(b(\Phi^\epsilon+S^0, P^\epsilon+p^0))^{-1}\|_{L^\infty}\langle b(\Phi^\epsilon+S^0, P^\epsilon+p^0)  \Phi^\epsilon, \Phi^\epsilon\rangle\nonumber\\
\leq &C_0 \big\{\langle a(\Phi^\epsilon+S^0, P^\epsilon+p^0)  P^\epsilon,  P^\epsilon\rangle+
\langle r(\Phi^\epsilon+S^0, P^\epsilon+p^0)  \U^\epsilon,  \U^\epsilon\rangle\nonumber\\
&+ \langle b(\Phi^\epsilon+S^0, P^\epsilon+p^0)  \Phi^\epsilon,  \Phi^\epsilon\rangle\big\},
\end{align}
since $a(\Phi^\epsilon+S^0, P^\epsilon+p^0)$ and $r(\Phi^\epsilon+S^0, P^\epsilon+p^0)$ are uniformly bounded away from
zero.
%
%

Multiplying \eqref{error3} by ${ }\F^\epsilon$ and \eqref{error4} by $\G^\epsilon$ respectively,
and integrating them over $\mathbb{T}^3$, we find that
\begin{align}
  &  \frac12 \frac{\dif}{\dif t}(\| \sqrt{\epsilon }\,\F^\epsilon\| ^2+\| \G^\epsilon\| ^2)
   +\int (\cu \F^\epsilon\cdot \G^\epsilon-\cu \G^\epsilon\cdot \F^\epsilon)\dif x
   + \| \F^\epsilon\| ^2
   = \left\langle \mathbf{f}_4^\epsilon ,  \F^\epsilon\right\rangle.\label{L2M}
\end{align}
By the regularity of $ (\u^0,\H^0)$, Cauchy-Schwarz's inequality, and Sobolev's imbedding,
the terms on the right-hand side of \eqref{L2M} can be bounded by
\begin{align*}
 \frac{1}{4} \| \F^\epsilon\| ^2+C(\| \mathcal{E}^\epsilon(t)\| _{s}^2+1)\| (\U^\epsilon,\G^\epsilon)\| ^2+C\epsilon^2.
\end{align*}
From the fact that
\begin{align*}
   \int (\cu \F^\epsilon\cdot \G^\epsilon-\cu \G^\epsilon\cdot \F^\epsilon)\dif x=\int \dv(\F^\epsilon\times \G^\epsilon)\dif x =0,
\end{align*}
we get
\begin{align}
&  \frac12 \frac{d}{dt}(\|\sqrt{\epsilon }\F^\epsilon\| ^2+\| \G^\epsilon\| ^2)
   +\frac{3  }{4}\| \F^\epsilon\| ^2\nonumber\\
& \qquad \qquad \qquad   \leq
  C(\| \mathcal{E}^\epsilon(t)\| _{s}^2+1)\| (\U^\epsilon,\G^\epsilon)\| ^2+C\epsilon^2. \label{L2M2}
\end{align}
Note that here we have used the special structure of \eqref{error3} and \eqref{error4}.

 Thus, we can easily obtain \eqref{L2} by integrating \eqref{L2M}  and  \eqref{L2M2} over $[0,T]$ and then combining the result with \eqref{we} and \eqref{wf}.
\end{proof}

\medskip
\subsection{Higher order estimates on $\Phi^\epsilon, \F^\epsilon,$ and $\G^\epsilon$}

In order to close the estimate \eqref{L2}, we need to derive the higher order estimates of the system  \eqref{error1}--\eqref{error4}. We first consider
the estimates on $\Phi^\epsilon, \F^\epsilon,$ and $\G^\epsilon$.

\begin{lem}\label{LHa}
 Let the assumptions in Proposition \ref{P31} hold and the multi-index $\alpha$ satisy $1\leq |\alpha|\leq s$. Then,
  for all $0<t<T $ and sufficiently small $\epsilon$, we have
   \begin{align}\label{H2a}
     &\|(\partial^\alpha_x \Phi^\epsilon,\partial^\alpha_x\G^\epsilon)\|^2 +\epsilon\|\partial^\alpha_x \F^\epsilon\|^2
   +  \frac{3}{2} \int^t_0
\| \pa\F^\epsilon(\tau)\| ^2\dif\tau\nonumber\\
     \leq  &  \big\{\|(\partial^\alpha_x \Phi^\epsilon,\partial^\alpha_x\G^\epsilon)\|^2 +\epsilon\|\partial^\alpha_x \F^\epsilon\|^2\big\} (t=0) +C_T\epsilon^2\nonumber\\
   &   + C\int^t_0\Big
     \{\gamma_1 \| \F^\epsilon\| ^4_{s} +\left(\gamma_2+\gamma_3\right)\| \F^\epsilon\| ^2_{s-1}+\|\mathcal{E}\|^{2}_s+\mathcal{E}\|^{2s}_s\nonumber\\
     & + {C}(1+\|\mathcal{E}\|^{2(s+1)}_s)
 \|(\partial^\alpha_x \Phi^\epsilon,  \partial^\alpha_xP^\epsilon, \partial^\alpha_x\U^\epsilon, \partial^\alpha_x\G^\epsilon)\|^2
 \Big\}(\tau)\dif \tau,
          \end{align}
            where  $\gamma_1$, $\gamma_2$, and $\gamma_3$ are sufficiently small positive constants.
\end{lem}
\begin{proof}
  Dividing \eqref{error22} by $b(\Phi^\epsilon+S^0, P^\epsilon+p^0)$, applying  operator  $\partial^\alpha_x\, (1\leq|\alpha|\leq s)$ to the resulting equation,
   multiplying
    by $ \partial^\alpha_x \Phi^\epsilon$, and integrating over $\mathbb{T}^3$, we obtain that
    \begin{align}\label{zn1}
  \frac12\frac{\rm d}{{\rm d}t}\left\langle  \pa  \Phi ^\epsilon, \pa  \Phi^\epsilon \right\rangle
 = & -\left\langle \pa((\U^\epsilon+\u^0)\cdot\nabla  \Phi^\epsilon),\pa  \Phi^\epsilon\right\rangle\nonumber\\
 &+\left\langle \partial^\alpha_x\left\{\frac{\mathbf{f}^\epsilon_3}{b(\Phi^\epsilon+S^0, P^\epsilon+p^0)}\right\},\pa  \Phi^\epsilon \right\rangle.
    \end{align}
Now we bound the terms on the right-hand side of \eqref{zn1}. By the regularity of $\u^0$,
Cauchy-Schwarz's inequality, and Sobolev's imbedding, we see that
\begin{align}\label{zn2}
 & \ \ \ \ \langle \partial^\alpha_x([(\U^\epsilon+\u^0)\cdot \nabla] \Phi^\epsilon),\partial^\alpha_x  \Phi^\epsilon\rangle\nonumber\\
& = \langle [(\U^\epsilon+\u^0)\cdot \nabla]\partial^\alpha_x  \Phi^\epsilon , \partial^\alpha_x  \Phi^\epsilon\rangle\
  +\big\langle \mathcal{H}^{(1)},\partial^\alpha_x  \Phi^\epsilon \big\rangle\nonumber\\
  & = -\frac12 \langle \dv(\U^\epsilon+\u^0) \partial^\alpha_x  \Phi^\epsilon , \partial^\alpha_x  \Phi^\epsilon \rangle\
   +\big\langle \mathcal{H}^{(1)},\partial^\alpha_x  \Phi^\epsilon \big \rangle\nonumber\\
 & \leq C(\| \mathcal{E}^\epsilon(t)\| _{s}+1)\| \partial^\alpha_x  \Phi^\epsilon\| ^2   +\|  \mathcal{H}^{(1)}\| ^2,
\end{align}
where the commutator
\begin{align*}
 \mathcal{H}^{(1)}  :=\partial^\alpha_x([(\U^\epsilon+\u^0)\cdot \nabla] \Phi^\epsilon)-[(\U^\epsilon+\u^0)\cdot \nabla]\partial^\alpha_x  \Phi^\epsilon .
\end{align*}
We use the Moser-type and Cauchy-Schwarz's inequalities,  the regularity of $\u^0$, and Sobolev's imbedding to infer that
\begin{align}\label{znc}
  \big\|\mathcal{H}^{(1)} \big\| &\leq C( \| D_x^1(\U^\epsilon+\u^0)\| _{L^\infty}\| D_x^s \Phi^\epsilon\|
+\| D_x^1 \Phi^\epsilon\| _{L^\infty}\| D^{s-1}_x(\U^\epsilon+\u^0)\|) \nonumber\\
   & \leq C\| \mathcal{E}^\epsilon(t)\| _{s}^2+C\| \mathcal{E}^\epsilon(t)\| _{s}.
\end{align}

By the definiton of $\mathbf{f}^\epsilon_3$,  the last term in \eqref{zn1} can be rewritten as
 \begin{align}\label{HT1}
  &\left\langle \partial^\alpha_x\left\{\frac{\mathbf{f}^\epsilon_3}{b(\Phi^\epsilon+S^0, P^\epsilon+p^0)}\right\},\pa  \Phi^\epsilon \right\rangle\nonumber\\
  =  &   -\left\langle \partial^\alpha_x\left\{\frac{[b(\Phi^\epsilon+S^0, P^\epsilon+p^0)-b(S^0,p^0)][\partial_t S^0 +\u^0\cdot\nabla S^0]}{b(\Phi^\epsilon+S^0, P^\epsilon+p^0)}
    \right\},\pa\Phi^\epsilon\right\rangle\nonumber\\
   &  -\left\langle \partial^\alpha_x (\U^\epsilon\cdot \nabla S^0) ,\pa\Phi^\epsilon\right\rangle
      +\left\langle\pa\left\{ \frac{|\F^\epsilon+\U^\epsilon\times \G^\epsilon|^2}{b(\Phi^\epsilon+S^0, P^\epsilon+p^0)}\right\},\pa\Phi^\epsilon\right\rangle\nonumber\\
    & +\left\langle\pa\left\{ \frac{|\u^0\times \G^\epsilon+\U^\epsilon\times \H^0|^2}{b(\Phi^\epsilon+S^0, P^\epsilon+p^0)}\right\},\pa\Phi^\epsilon\right\rangle\nonumber\\
    &  + \left\langle\pa\left\{\frac{2\F^\epsilon}{b(\Phi^\epsilon+S^0, P^\epsilon+p^0)}\cdot
                   [\cu \H^0+\u^0\times \G^\epsilon+\U^\epsilon\times \H^0]\right\},\pa\Phi^\epsilon\right\rangle\nonumber\\
    &  + \left\langle\pa\left\{\frac{2(\U^\epsilon\times \G^\epsilon)}{b(\Phi^\epsilon+S^0, P^\epsilon+p^0)}\cdot
                   [\cu \H^0+\u^0\times \G^\epsilon+\U^\epsilon\times \H^0]\right\},\pa\Phi^\epsilon\right\rangle\nonumber\\
      &  +   \left\langle\pa\left\{\frac{2}{b(\Phi^\epsilon+S^0, P^\epsilon+p^0)}\cu \H^0\cdot (\u^0\times \G^\epsilon+\U^\epsilon\times \H^0)\right\},\pa\Phi^\epsilon\right\rangle\nonumber\\
     : = & \sum^{7}_{i=1}\mathcal{I}^{(i)}.
\end{align}
We have to bound the terms on the right-hand side of \eqref{HT1}.
 By  the regularity of $ ( S^0,p^0,\u^0)$, the positivity of $b(\Phi^\epsilon+S^0, P^\epsilon+p^0)$, \eqref{ho},  and Cauchy-Schwarz's inequality, the
 term $\mathcal{I}^{(1)}$   can be bounded as follows
\begin{align}\label{ht8}
   \big|\mathcal{I}^{(1)} \big| \leq  C(\|\mathcal{E}^\epsilon(t)\|_{s}^{2s} +1)
   ( \|\partial^\alpha_x \Phi^\epsilon\|^{2} +\|\partial^\alpha_x P^\epsilon\|^2).
\end{align}
Similarly, the    term  $\mathcal{I}^{(2)}$  can be controlled by
\begin{align}\label{ht88}
  \big| \mathcal{I}^{(2)} \big|  \leq  C(\|\U^\epsilon(t)\|_{s}^2   +   \|\partial^\alpha_x\Phi^\epsilon\|^2).
\end{align}
For the  term $\mathcal{I}^{(3)}$, we rewrite it as
\begin{align*}
\mathcal{I}^{(3)} &=  \left\langle\pa\left\{ \frac{1}{b(\Phi^\epsilon+S^0, P^\epsilon+p^0)}|\F^\epsilon+\U^\epsilon\times \G^\epsilon|^2\right\},\pa\Phi^\epsilon\right\rangle\nonumber\\
 &=  \left\langle\pa\left\{ \frac{1}{b(\Phi^\epsilon+S^0, P^\epsilon+p^0)}|\F^\epsilon|^2\right\},\pa\Phi^\epsilon\right\rangle\nonumber\\
  &\ \ \ \ +  \left\langle\pa\left\{ \frac{2}{b(\Phi^\epsilon+S^0, P^\epsilon+p^0)}\F^\epsilon\cdot(\U^\epsilon\times \G^\epsilon)\right\},\pa\Phi^\epsilon\right\rangle\nonumber\\
   &\ \ \ \  +  \left\langle\pa\left\{ \frac{1}{b(\Phi^\epsilon+S^0, P^\epsilon+p^0)} |\U^\epsilon\times \G^\epsilon|^2\right\},\pa\Phi^\epsilon\right\rangle\nonumber\\
    &: =\mathcal{I}^{(3_1)}+\mathcal{I}^{(3_2)}+\mathcal{I}^{(3_3)}.
      \end{align*}
By Cauchy-Schwarz's inequality and Sobolev's embedding, the term $\mathcal{I}^{(3_1)}$ can be bounded by
\begin{align}\label{ht111}
 \mathcal{I}^{(3_1)} =& \left\langle \frac{1}{b(\Phi^\epsilon+S^0, P^\epsilon+p^0)}\pa\left(|\F^\epsilon|^2\right),\pa\Phi^\epsilon\right\rangle\nonumber\\
 & +\sum_{\beta \leq \alpha,|\beta|<|\alpha|}
 \left\langle\partial_x^{\alpha-\beta}\left( \frac{1}{b(\Phi^\epsilon+S^0, P^\epsilon+p^0)}\right)\partial_x^{\beta}(|\F^\epsilon|^2),\pa\Phi^\epsilon\right\rangle\nonumber\\
 \leq & \gamma_1 \|\F^\epsilon\|^4_s+C_{\gamma_1}\|\pa\Phi^\epsilon\|^2(1+\|\mathcal{E}(t)\|^{2(s+1)}_s)
\end{align}
for any $\gamma_1>0$. For  the term $\mathcal{I}^{(3_2)}$, by the positivity of $b(\Phi^\epsilon+S^0, P^\epsilon+p^0)$ and
Sobolev's imbedding,   we have
\begin{align}\label{ht12}
   \mathcal{I}^{(3_2)}& = 2\left\langle\partial^\alpha_x \F^\epsilon\cdot
\frac{\U^\epsilon\times\G^\epsilon }{b(\Phi^\epsilon+S^0, P^\epsilon+p^0)},
   \partial^\alpha_x\Phi^\epsilon \right\rangle
+  2\big\langle \mathcal{H}^{(2)},
\partial^\alpha_x\Phi^\epsilon
\big\rangle \nonumber\\
 & \leq \frac{1}{16} \| \partial^\alpha_x \F^\epsilon\| ^2+
 C\| \mathcal{E}^\epsilon(t)\| _{s}^2
 \| \partial^\alpha_x\U^\epsilon\| ^2+2 \big\langle \mathcal{H}^{(2)},  \partial^\alpha_x\Phi^\epsilon
\big\rangle,
\end{align}
where the commutator
\begin{align*}
 \mathcal{H}^{(2)}:=\partial^\alpha_x\left\{
\F^\epsilon\cdot
\frac{\U^\epsilon\times\G^\epsilon }{b(\Phi^\epsilon+S^0, P^\epsilon+p^0)}\right\}-\partial^\alpha_x \F^\epsilon\cdot
\frac{\U^\epsilon\times\G^\epsilon }{b(\Phi^\epsilon+S^0, P^\epsilon+p^0)}.
\end{align*}
By the Cauchy-Schwarz's and Moser-type  inequalities, we obtain that
 \begin{align}\label{ht13}
 &  \ \ \   2\big|\big\langle\mathcal{H}^{(2)},
  \partial^\alpha_x\Phi^\epsilon\big\rangle\big| \leq  2\| \mathcal{H}^{(2)}\| \cdot \| \partial^\alpha_x\Phi^\epsilon \|  \nonumber\\
  & \leq C \bigg[\left\Vert D_x^1\left(\frac{\U^\epsilon\times\G^\epsilon }{b(\Phi^\epsilon+S^0, P^\epsilon+p^0)}\right)\right\Vert_{L^\infty}\| \F^\epsilon\| _{s-1}
  \nonumber\\
   &\ \ \ \ +\| \F^\epsilon\| _{L^\infty}\left\Vert\frac{\U^\epsilon\times\G^\epsilon }{b(\Phi^\epsilon+S^0, P^\epsilon+p^0)}\right\Vert_{s}\bigg]\| \partial^\alpha_x\Phi^\epsilon \| \nonumber\\
   &\leq  \gamma_2 \| \F^\epsilon\| ^2_{s-1}+C_{\gamma_2} (\| \mathcal{E}^\epsilon(t)\| ^2_{s}+1) \| \partial^\alpha_x
   \Phi^\epsilon\| ^2.
 \end{align}
for any $\gamma_2>0$. For the term  $\mathcal{I}^{(3_3)}$, by Cauchy-Schwarz's and the Moser-type inequalities, it can be bounded by
\begin{align}\label{ht133}
\big|\mathcal{I}^{(3_3)} \big|  \leq & C(1+\|\mathcal{E}(t)\|^{2(s+1)}_s)
\|(\partial^\alpha_x \Phi^\epsilon,  \partial^\alpha_xP^\epsilon, \partial^\alpha_x\U^\epsilon, \partial^\alpha_x\G^\epsilon)\|^2.
\end{align}
 By  the regularity of $  (S^0,\u^0, \H^0)$,   the positivity of $b(\Phi^\epsilon+S^0, P^\epsilon+p^0)$, and Cauchy-Schwarz's inequality, the
 terms $\mathcal{I}^{(4)}$ and  $\mathcal{I}^{(7)}$   can be bounded as follows:
\begin{align}\label{ht14}
  \big| \mathcal{I}^{(4)} \big|+\big|\mathcal{I}^{(7)} \big| \leq  C(\|\mathcal{E}^\epsilon(t)\|_{s}^{2s} +1)
   \|(\partial^\alpha_x \Phi^\epsilon,  \partial^\alpha_xP^\epsilon, \partial^\alpha_x\U^\epsilon, \partial^\alpha_x\G^\epsilon)\|^2.
\end{align}
The term $ \mathcal{I}^{(5)}$ can be bounded, similarly to $\mathcal{I}^{(3_2)}$, by
\begin{align}\label{ht15}
\big|  \mathcal{I}^{(5)}\big|\leq   \gamma_3 \| \F^\epsilon\| ^2_{s-1}+C_{\gamma_3} (\| \mathcal{E}^\epsilon(t)\| ^2_{s}+1)
   \|(\partial^\alpha_x \Phi^\epsilon,  \partial^\alpha_xP^\epsilon, \partial^\alpha_x\U^\epsilon, \partial^\alpha_x\G^\epsilon)\|^2
\end{align}
for any $\gamma_3>0$. Finally, similar to $\mathcal{I}^{(3_3)}$, the term $ \mathcal{I}^{(6)}$ can be controlled by
\begin{align}\label{ht16}
 \big| \mathcal{I}^{(6)} \big|  \leq & C(1+\|\mathcal{E}(t)\|^{2(s+1)}_s)
 \|(\partial^\alpha_x \Phi^\epsilon,  \partial^\alpha_xP^\epsilon, \partial^\alpha_x\U^\epsilon, \partial^\alpha_x\G^\epsilon)\|^2.
\end{align}

Substituting  \eqref{zn2}--\eqref{ht16} into \eqref{zn1}, we conclude that
 \begin{align}\label{HT202a}
\!\!\!  \frac12\frac{\dif}{\dif t} \langle \partial^\alpha_x\Phi^\epsilon,
\partial^\alpha_x\Phi^\epsilon \rangle
\leq   &
    {C}_{{\gamma}}(1+\|\mathcal{E}(t)\|^{2(s+1)}_s)
 \|(\partial^\alpha_x \Phi^\epsilon,  \partial^\alpha_xP^\epsilon, \partial^\alpha_x\U^\epsilon, \partial^\alpha_x\G^\epsilon)\|^2
 \nonumber\\
  & +\gamma_1 \| \F^\epsilon\| ^4_{s} +\left(\gamma_2+\gamma_3\right)\| \F^\epsilon\| ^2_{s-1} +C\|\mathcal{E}(t)\|^{2s}_s+
  C\|\mathcal{E}(t)\|^{2}_s.
\end{align}
for some constant $ {C}_{{\gamma}}>0$ depending on $\gamma_j$ ($j=1,2,3$).

Applying operator
$\partial^\alpha_x\, (1\leq|\alpha|\leq s)$ to \eqref{error3} and \eqref{error4}, multiplying  the resulting equations
    by ${ }\partial^\alpha_x\F^\epsilon$ and $\partial^\alpha_x\G^\epsilon$ respectively, and then integrating    over
    ${\mathbb T}^3$, we  obtain that
\begin{align}
  &  \frac12\frac{\dif}{\dif t}(\epsilon\|\partial^\alpha_x\F^\epsilon\| ^2
  +\| \partial^\alpha_x\G^\epsilon\| ^2) + \| \partial^\alpha_x\F^\epsilon\| ^2\nonumber\\
 & +\int (\cu \partial^\alpha_x\F^\epsilon\cdot \partial^\alpha_x \G^\epsilon-\cu \partial^\alpha_x\G^\epsilon\cdot \partial^\alpha_x\F^\epsilon)\dif x\nonumber\\
 =  &
      \left\langle [\partial^\alpha_x(\U^\epsilon\times \H^0)+\partial^\alpha_x(\u^0\times \G^\epsilon)]
     -\partial^\alpha_x(\U^\epsilon\times \G^\epsilon),
    \partial^\alpha_x\F^\epsilon\right\rangle\nonumber\\
 &\ -  \left\langle{\epsilon}\partial^\alpha_x\partial_t
 \cu \H^0+\epsilon\partial^\alpha_x\partial_t(\u^0\times \H^0),
 \partial^\alpha_x\F^\epsilon\right\rangle. \label{L2MMa}
\end{align}
In view of the regularity of $ (\u^0,\H^0)$, Cauchy-Schwarz's and the
Moser-type inequalities, and Sobolev's imbedding, the
terms on the right-hand side of \eqref{L2MMa} can be controlled by
\begin{align}\label{L2M1a}
 \frac{1}{4} \|\partial^\alpha_x\F^\epsilon\|^2
 +C(\|\partial^\alpha_x\mathcal{E}^\epsilon(t)\|_{s}^2+1)\|(\partial^\alpha_x\U^\epsilon,
 \partial^\alpha_x\G^\epsilon)\|^2+C\epsilon^2.
\end{align}
 Noticing the fact that
\begin{align*}
   \int (\cu \partial^\alpha_x\F^\epsilon\cdot \partial^\alpha_x\G^\epsilon
   -\cu \partial^\alpha_x\G^\epsilon\cdot \partial^\alpha_x\F^\epsilon)\dif x
   =\int \dv(\partial^\alpha_x\F^\epsilon\times \partial^\alpha_x\G^\epsilon)\dif x =0,
\end{align*}
we find that
\begin{align}
 &\frac12  \frac{\dif}{\dif t}(\sqrt{\epsilon }\,\partial^\alpha_x\F^\epsilon\| ^2+\| \partial^\alpha_x\G^\epsilon\| ^2)
   +\frac{3  }{4}\| \partial^\alpha_x\F^\epsilon\| ^2  \nonumber\\
  &\qquad  \leq
  C(\| \mathcal{E}^\epsilon(t)\| _{s}^2+1)\| (\partial^\alpha_x\U^\epsilon,\partial^\alpha_x\G^\epsilon)\| ^2+C\epsilon^2. \label{L2M2ba}
\end{align}
We should point out that here we have used the special structure of \eqref{error3} and \eqref{error4}.

Combining \eqref{HT202a} with \eqref{L2M2ba}, one obtains the estimate \eqref{H2a}.
\end{proof}


\medskip
\subsection{Higher order estimates on $P^\epsilon$ and  $\U^\epsilon $}

In order to close the estimates \eqref{L2} and \eqref{H2a}, now we study the higher order derivatives of $P^\epsilon$ and $\U^\epsilon$.
We shall adapt the techniques developed in \cite{MS01,JJL4}. Set
\begin{align*}
  \mathcal{A}(\Phi^\epsilon, P^\epsilon)& =\left(\begin{array}{cc}
                   a(\Phi^\epsilon+S^0, P^\epsilon+p^0) & 0 \\
                    0 & r(\Phi^\epsilon+S^0, P^\epsilon+p^0)\mathbf{I}_{3}
                   \end{array}\right),\nonumber\\
   \mathcal{L}(\partial_x)& =\left(\begin{array}{cc}
                    0 & \dv  \\
                    \nabla & 0 \end{array}\right), \qquad
    \mathcal{U}^\epsilon=\left(\begin{array}{c}
                   P^\epsilon\\
                   \U^\epsilon
                    \end{array}
                    \right).
                    \end{align*}

 Let $\mathcal{L}_{\mathcal{A}}(\partial_x ):=  \{ \mathcal{A}(\Phi^\epsilon, P^\epsilon)\} ^{-1}\mathcal{L}(\partial_x)$. We have

\begin{lem}\label{LHB}
  There are constants $K>0$ and $C_1>0$ such that  for all $\sigma\in\{1,\dots,s\}$,
$\epsilon \in (0,1]$ and $t\in [0,T]$,
\begin{align}
 & \|\mathcal{U}^\epsilon\|_\sigma \leq K\big\{\|\mathcal{L}(\partial_x) {\mathcal{U}^\epsilon}\|_{\sigma -1}
+ \|\cu \U^\epsilon  \|_{\sigma-1}+\|\mathcal{U}^\epsilon\|_{\sigma-1}\big\},\label{pua}\\
&\|\mathcal{U}^\epsilon\|_\sigma \leq C_1\big\{\|\{\mathcal{L}_{\mathcal{A}}(\partial_x)\}^\sigma  \mathcal{U}^\epsilon\|_{0}
+\|\cu \U^\epsilon\|_{\sigma-1}+\|\mathcal{U}^\epsilon\|_{\sigma-1}\big\}.\label{pub}
\end{align}
\end{lem}
\begin{proof} \eqref{pua} is obvious. Recalling the fact that
$ a(\Phi^\epsilon+S^0, P^\epsilon+p^0)$ and $ r(\Phi^\epsilon+S^0, P^\epsilon+p^0)$ are smooth, positive, and bounded away from zero
with respect to each $\epsilon$, and applying the inequality
$$ \|\mathcal{L}(\partial_x) \mathcal{U}^\epsilon\|\leq \|\mathcal{A}\|_{L^\infty}\|\mathcal{L}_{\mathcal{A}}(\partial_x )\mathcal{U}^\epsilon\|,$$
we see that \eqref{pub} can be shown easily by induction on $\sigma$.
\end{proof}

Next, we  bound $\|\{\mathcal{L}_{\mathcal{A}}(\partial_x)\}^\sigma  {\mathcal{U}}^\epsilon\|_{0}$ and
$\|\cu \U^\epsilon\|_{\sigma-1}$ by induction. We first show the following estimate.

\begin{lem}\label{LHC}
  There exist constants $\bar s>0$ and $\kappa_1>0$, such that for all $1\leq \sigma \leq s$, all
$\epsilon \in (0,1]$ and $t\in [0,T]$, it holds that
\begin{align}
 &\int  \big(\mathcal{A}(\Phi^\epsilon, P^\epsilon)|\{\mathcal{L}_{\mathcal{A}}(\partial_x)\}^\sigma  {\mathcal{U}}^\epsilon|^2\big)(t)\,\textrm{\emph{d}}x
\leq \int  \big(\mathcal{A}(\Phi^\epsilon, P^\epsilon)|\{\mathcal{L}_{\mathcal{A}}(\partial_x)\}^\sigma  {\mathcal{U}}^\epsilon|^2\big)(0)\,\textrm{\emph{d}}x
\nonumber\\
&\quad
+C\epsilon^2  + \int^t_0 \kappa_1 \|\F^\epsilon(\tau)\|^4_s\dif \tau
+C_{\kappa_1}\int^t_0(1+\|\mathcal{E}(\tau)\|^{2\bar s}_s)\|\{\mathcal{L}_{\mathcal{A}}(\partial_x)\}^\sigma \mathcal{U}^\epsilon(\tau)\|^2\dif \tau. \label{puc}
\end{align}
\end{lem}

\begin{proof} We  follow the arguments in \cite{MS01,JJL4} with modifications. For simplicity, we set
 $\mathcal{A}:= \mathcal{A}(\Phi^\epsilon, P^\epsilon)$.
 Let $\mathcal{U}_\sigma^\epsilon:=\{\mathcal{L}_{\mathcal{A}}(\partial_x)\}^\sigma \mathcal{U}^\epsilon$,
$\sigma\in \{1,\dots, s\}.$
It is easy to verify that the operator $\mathcal{L}_{\mathcal{A}}(\partial_x )$ is
bounded from $H^{\sigma}$ to $H^{\sigma-1}$ for $\sigma \in \{1,\dots, s\}$. Note that
the equations \eqref{error1} and \eqref{error2} can be written as
\begin{align}\label{pud}
  (\partial_t+(\U^\epsilon+\u^0)\cdot \nabla)\mathcal{U}^\epsilon+\frac{1}{\epsilon}\mathcal{A}^{-1}\mathcal{L}(\partial_x)\mathcal{U}^\epsilon
  =  \mathcal{A}^{-1}\mathcal{J}^\epsilon, \quad
 {\mathcal{J}}^\epsilon=\left(\begin{array}{c}
                 f^\epsilon_1 \\
                \mathbf{f}^\epsilon_2
                 \end{array}
                 \right).
\end{align}
 For $ \sigma \geq 1$, we commute the operator $\{\mathcal{L}_{\mathcal{A}}\}^\sigma$ with  \eqref{pud} and
 multiply the resulting system by $\mathcal{A}$ to infer that
\begin{align}\label{pue}
  \mathcal{A}(\partial_t+(\U^\epsilon+\u^0)\cdot \nabla)\mathcal{U}^\epsilon_\sigma
  +\frac{1}{\epsilon}\mathcal{L}(\partial_x)\mathcal{U}^\epsilon_\sigma=\mathcal{A}({g}^\epsilon_\sigma+
  \mathbf{h}^\epsilon_\sigma),
\end{align}
where
$$ {g}^\epsilon_\sigma:= [\partial_t+(\U^\epsilon+\u^0)\cdot \nabla , \{\mathcal{L}_{\mathcal{A}}\}^\sigma] \mathcal{U}^\epsilon,\qquad
\mathbf{h}^\epsilon_\sigma:= \{\mathcal{L}_{\mathcal{A}}\}^\sigma (\mathcal{A}^{-1} \mathcal{J}^\epsilon). $$
Multiplying \eqref{pue} by $ \mathcal{U}^\epsilon_\sigma$ and integrating
over $(0,t)\times\mathbb{T}^3$ with $t\leq T$, and noticing that
 the singular terms cancel out since $\mathcal{L}(\partial_x)$ is skew-adjoint, we then use
the inequality \eqref{wb}  and Cauchy-Schwarz's
inequality to deduce that
\begin{align}\label{puf}
  \frac{1}{2}\langle \mathcal{A}(t)\mathcal{U}^\epsilon_\sigma(t), \mathcal{U}^\epsilon_\sigma(t)\rangle
\leq & \frac{1}{2}\langle \mathcal{A}(0)\mathcal{U}^\epsilon_\sigma(0),
\mathcal{U}^\epsilon_\sigma(0)\rangle
         +C\int^t_0(1+\|\mathcal{E}\|_s^4(\tau)) \|\mathcal{U}^\epsilon_\sigma(\tau)\|^2 \textrm{d}\tau \nonumber\\
  & +C\epsilon^2 +  \int^t_0 (\|{g}^\epsilon_\sigma(\tau)\|^2_\sigma+\|\mathbf{h}^\epsilon_\sigma(\tau)\|^2_\sigma)  \|\mathcal{U}^\epsilon_\sigma(\tau)\|^2 \dif \tau.
\end{align}
Following the proof process of Lemma 2.4 in \cite{MS01} and applying \eqref{wc}, we obtain that
\begin{align}\label{pug}
 \|{g}^\epsilon_\sigma(t)\|\leq C(1+\|\mathcal{E}(t)\|^{2s_1}_s)
\end{align}
for some constant $s_1>0$.

Now we estimate the   term   $\mathbf{h}_\sigma^\epsilon$.
Noticing that  there is no spatial derivatives of
$( P^\epsilon,\U^\epsilon,\Phi^\epsilon,\F^\epsilon,\G^\epsilon)$ in the definitions of $f^\epsilon_1$ and $\mathbf{f}^\epsilon_2$,
thus  we can apply the regularity of $(p^0,\u^0,S^0,\H^0)$, \eqref{ma} and \eqref{wc} to obtain that
\begin{align}\label{puh}
 \|\mathbf{h}^\epsilon_\sigma(t)\|\leq \kappa_1\|\F^\epsilon(t)\|^2_s+C_{\kappa_1}(1+\|\mathcal{E}(t)\|^{2s_1}_s)
\end{align}
for some constant $s_2>0$ and sufficient small $\kappa_1>0$.

Putting   \eqref{pug} and \eqref{puh} into \eqref{puf} and choosing $\bar s=\max\{2,s_1,s_2\}$, we get \eqref{puc}.
\end{proof}

Finally, we derive an estimate for
$\|\cu\U^\epsilon\|_{\sigma-1}$.
Dividing \eqref{error2} by $ r(\Phi^\epsilon+S^0, P^\epsilon+p^0)$ and applying the operator \emph{curl} to the resulting equations, we obtain that
\begin{align}
& [\partial_t +((\U^\epsilon+\u^0)\cdot \nabla)](\cu\U^\epsilon)=\, [(\U^\epsilon+\u^0)\cdot \nabla, \cu]\U^\epsilon  \nonumber\\
 &\qquad
 -\cu  \left(\frac{\nabla \Phi^\epsilon}{r(\Phi^\epsilon+S^0, P^\epsilon+p^0)}\right)
 +\cu \left(\frac{\mathbf{f}^\epsilon_2}{r(\Phi^\epsilon+S^0, P^\epsilon+p^0)}\right).  \label{pui}
 \end{align}
\begin{lem}\label{LHD} There exist constants $\bar{\bar s}>0$ and $\kappa_2>0$, such that the following inequality holds:
\begin{align}  \label{puj}
\|\cu\U^\epsilon(t)\|^2_{s-1} \leq & \|\cu\U^\epsilon(0)\|^2_{s-1}+  C \epsilon^2 + \int^t_0 \kappa_2 \|\F^\epsilon(\tau)\|^4_s\dif \tau\nonumber\\
&
+C_{\kappa_2}\int^t_0(1+\|\mathcal{E}(\tau)\|^{2\bar{\bar s}}_s)\|\cu\U^\epsilon(\tau)\|_{s-1}^2\dif \tau.
 \end{align}
\end{lem}

\begin{proof}
Set $\omega^\epsilon:= \cu\U^\epsilon$.   Taking $\partial^\alpha_x$
$(0\leq|\alpha|\leq s-1)$ to \eqref{pui}, multiplying the resulting equations by $\omega^\epsilon$, and
integrating over $(0,t)\times \mathbb{T}^3$ with $t\leq T$, we infer that
\begin{align}\label{puk}
\frac{1}{2}\langle \partial^\alpha_x \omega^\epsilon(t),   \partial^\alpha_x \omega^\epsilon(t)\rangle
\leq\,&\frac{1}{2}\langle \partial^\alpha_x \omega^\epsilon(0),   \partial^\alpha_x\omega^\epsilon(0)\rangle
+ C\int^t_0(1+\|\mathcal{E}\|_s^4(\tau))\|\partial^\alpha \omega^\epsilon(\tau)\|^2\textrm{d}\tau\nonumber \\
&+\int^t_0\left\langle  [(\U^\epsilon+\u^0)\cdot \nabla, \partial^\alpha_x]\omega^\epsilon ,\partial^\alpha_x \omega^\epsilon\right\rangle\dif\tau  \nonumber\\
 &
 -\int^t_0\left\langle \partial^\alpha_x\cu  \left(\frac{\nabla \Phi^\epsilon}{r(\Phi^\epsilon+S^0, P^\epsilon+p^0)}\right),\partial^\alpha_x \omega^\epsilon\right\rangle \dif\tau \nonumber\\
&+\int^t_0\left\langle\partial^\alpha_x\cu \left(\frac{\mathbf{f}^\epsilon_2}{r(\Phi^\epsilon+S^0, P^\epsilon+p^0)}\right),\partial^\alpha_x \omega^\epsilon\right\rangle\dif\tau\nonumber\\
 :=\,&\frac{1}{2}\langle \partial^\alpha_x \omega^\epsilon(0),   \partial^\alpha_x\omega^\epsilon(0)\rangle
+ C\int^t_0(1+\|\mathcal{E}(\tau)\|_s^4)\|\partial^\alpha_x \omega^\epsilon(\tau)\|^2\textrm{d}\tau\nonumber \\
& +  \int^t_0\sum_{i=1}^3 \mathcal{N}_i(\tau)\dif \tau.
 \end{align}

Next, we estimate the terms $\mathcal{N}_i(\tau)$ ($i=1,2,3$) on the right-hand side of \eqref{puk}.
From Cauchy-Schwarz's inequality we get that
\begin{align*}
 | \mathcal{N}_1(\tau)|\leq   C\|\partial^\alpha_x\omega^\epsilon(\tau)\| \, \|\hat{\mathbf{h}}^\epsilon_\alpha(\tau)\|,
 \qquad  \hat{\mathbf{h}}^\epsilon_\alpha(\tau):= [(\U^\epsilon+\u^0)\cdot \nabla, \partial_x^\alpha]\omega^\epsilon.
\end{align*}
The commutator $\hat{\mathbf{h}}^\alpha_\alpha$ is a sum of terms
$\partial^\beta_x (\U^\epsilon+\u^0)\cdot\partial^\zeta_x \omega^\epsilon$
with multi-indices $\beta$ and $\zeta$ satisfying $|\beta|+|\zeta|\leq s$, $|\beta|>0$,
and $|\zeta|>0$. Thus,
$$\|\hat{\mathbf{h}}^\epsilon_\alpha(\tau)\|\leq C(1+\|\mathcal{E}(\tau)\|^{s}_{s}),$$
 where the the following nonlinear Sobolev inequality has been used (see \cite{Ho97}): For all
$\alpha=(\alpha_1, \alpha_2 , \alpha_3)$, $\sigma\geq 0$, and $f,g\in H^{k+\sigma}(\mathbb{T}^3)$, $|\alpha|=k$,
it holds that
$$ \|[f,\partial^\alpha_x ]g\|_{{\sigma}}\leq  C_0(\|f\|_{W^{1,\infty}}\| g\|_{{\sigma+k-1}}
    +\| f\|_{{\sigma+k}}\|g\|_{L^{\infty}}).
$$
Hence, we have
 \begin{align}\label{pul}
 | \mathcal{N}_1(\tau)|\leq    C(1+\|\mathcal{E}(\tau)\|^{2s}_{s})\|\partial^\alpha_x  \omega^\epsilon(\tau)\|^2.
\end{align}

Noting that basic vector formulas  $\cu(\psi \u)=\psi \cu \u+\nabla \psi \times \u$ and $\cu \nabla \psi =0$, the term $N_2(\tau)$ can be
estimated by
\begin{align}\label{pull}
  | \mathcal{N}_2(\tau)|\leq & \left\|\partial^\alpha_x\left\{\nabla \left(\frac{1}{r(\Phi^\epsilon+S^0, P^\epsilon+p^0)}\right)\times \nabla\Phi^\epsilon\right\}\right\|
  \|\partial^\alpha_x \omega^\epsilon(\tau)\|\nonumber\\
  \leq& C(1+\|\mathcal{E}(\tau)\|^{s_3}_{s})\|\partial^\alpha_x  \omega^\epsilon(\tau)\|
\end{align}
for some $s_3>0$, where the properties $r(\cdot,\cdot)$, Proposition \ref{Pb}, Sobolev's imbedding, \eqref{mo}, and \eqref{wc} have been used.

By   the definition of  $\mathbf{f}^\epsilon_2$,
 the regularity of $(p^0,\u^0,S^0,\H^0)$, Cauchy-Schwarz's
inequality, \eqref{ma},   \eqref{wc}, and \eqref{mo}, we have
\begin{align}\nonumber
 \left\|\partial^\alpha_x\cu \left(\frac{\mathbf{f}^\epsilon_2}{r(\Phi^\epsilon+S^0, P^\epsilon+p^0)}\right)(t)\right\|\leq \tilde\kappa_2\|\F^\epsilon(t)\|^2_s+C_{\tilde\kappa_2}(1+\|\mathcal{E}(t)\|^{2s_4}_s)
\end{align}
for some constant $s_4>0$ and sufficient small $\tilde\kappa_2>0$.
Thus the term $\mathcal{N}_3(\tau)$ can be bounded by
 \begin{align}\label{pun}
  | \mathcal{N}_3(\tau)|\leq  \big( \tilde\kappa_2\|\F^\epsilon(t)\|^2_s+C_{\tilde\kappa_2}(1+\|\mathcal{E}(t)\|^{s_4}_s)\big)\|\partial^\alpha  \omega^\epsilon(\tau)\| .
  \end{align}
Putting \eqref{puk}--\eqref{pun} together, summing up $\alpha$ with
$0\leq|\alpha|\leq s-1$,  applying Cauchy-Schwarz's
inequality, and choosing $\bar{\bar{s}}=\max\{s_3,s_4\}$ and some sufficient small $\kappa_2>0$, we   obtain \eqref{puj}.
\end{proof}

With the estimates in Lemmas   \ref{La}--\ref{LHD} in hand, we are in a position to prove Proposition \ref{P31}.

\begin{proof}[Proof of Proposition \ref{P31}]
As in \cite{JL,JL2,PW}, we introduce an $\epsilon$-weighted energy functional
$$
\Gamma^\epsilon(t) =  \v \mathcal{E}^\epsilon(t)\v ^2_{s}.
$$
Summing up  \eqref{H2a} with $1\leq|\alpha|\leq s$  and \eqref{puc}  with
$1\leq|\sigma|\leq s-1$, combining \eqref{L2} with \eqref{puj}, using \eqref{pub} and
the fact that $ a(\Phi^\epsilon+S^0, P^\epsilon+p^0)$ and $ r(\Phi^\epsilon+S^0, P^\epsilon+p^0)$ are smooth, positive, and bounded away from zero
with respect to each $\epsilon$, $\F^\epsilon \in C^l([0,T],H^{s-2l})\ (  l=0,1)$, and
$$\|\{\mathcal{L}_{\mathcal{A}}(\partial_x)\}^\sigma  {\mathcal{U}}^\epsilon\|_{0}\leq C(1+\|\mathcal{E}(\tau)\|^{s_5}_s)\|\U^\epsilon(\tau)\|_{s}$$
for some $s_5>0$, we can choose $\eta_i \,(i=1,2,3)$,
  $\gamma_i\,(i=1,2,3)$, and $\kappa_1,\kappa_2$ to be sufficiently small to deduce
  that there exist a sufficient large constant $s_0>0$ and
 a small $\epsilon_0> 0$ depending only on $T$, such that for any $\epsilon\in (0,\epsilon_0]$ and $t\in [0,T]$,
\begin{align}\label{gma}
\Gamma^\epsilon(t)\leq C\Gamma^\epsilon(t = 0)+ C\epsilon^2+C\int^t_0\Big\{
\big((1+(\Gamma^\epsilon)^{s_0}\big)\Gamma^\epsilon\Big\}(\tau)\dif \tau .
\end{align}
 Thus, applying Gronwall's lemma to \eqref{gma} with the assumption $\Gamma^\epsilon (t=0)\leq C\epsilon^2$
 and Proposition \ref{P31}, we obtain that there exist a $0<T_1<1$ and an $\epsilon>0$, such that $T^\epsilon\geq T_1$
for all $\epsilon\in (0,\epsilon]$ and
$\Gamma^\epsilon(t)\leq C\epsilon^2$ for all $ t \in [0,T_1]$. Therefore, the desired a priori estimate
\eqref{www} holds. Moreover, by the standard continuous induction method, we can extend
$T^\epsilon\geq T_0$ for any $T_0<T_*$.
\end{proof}

Now we prove Theorem \ref{th} by  applying  Proposition \ref{P31}.

\begin{proof}[Proof of Theorem \ref{th}]  By virtue of the definition of the error functions
$( P^\epsilon, \U^\epsilon,  \Phi^\epsilon,\linebreak \F^\epsilon,  \G^\epsilon)$,
the regularity of $( p^0,\u^0, S^0,\H^0)$, the error system \eqref{error1}--\eqref{error4} and the
primitive system \eqref{nca}--\eqref{nce} are equivalent on $[0,T]$ for some $T>0$.
Therefore  the assumption \eqref{ivda} in Theorem \ref{th} implies the assumption \eqref{ww} in Proposition
\ref{P31}, and hence \eqref{www} gives \eqref{iivda}.
\end{proof}


{\bf Acknowledgements:}
The authors  give their gratitude  to professor Hugo Beir\~{a}o da Veiga  for his valuable  constructive
suggestions  on  preparation this paper.
Jiang was supported by the National Basic Research Program under the Grant 2011CB309705
and NSFC (Grant Nos. 11229101, 11371065).
Li was supported by NSFC (Grant No. 11271184), NCET-11-0227,
PAPD, and the Fundamental Research Funds for the Central Universities.


\bibliographystyle{plain}

\end{document}